\documentclass[lettersize,journal]{IEEEtran}
\usepackage{amsmath,amsfonts}
\usepackage{algorithmic}
\usepackage{array}
\usepackage[caption=false,font=normalsize,labelfont=sf,textfont=sf]{subfig}
\usepackage{textcomp}
\usepackage{stfloats}
\usepackage{url}
\usepackage{verbatim}
\usepackage{graphicx}
        \usepackage{mdwmath}
        \usepackage{mdwtab}
        \usepackage{eqparbox}
\usepackage{amssymb}
\usepackage{amsthm,nccmath} 
\usepackage{cuted}
 \usepackage{tabularx,booktabs}
 \usepackage{algorithm,algorithmic}

\newtheorem{thm}{Theorem}[section]

\newtheorem{lem}[thm]{Lemma}

\newtheorem{defn}[thm]{Definition}

\newtheorem{assum}{Assumption}[section]

\newcommand{\cG}{\mathcal{G}}
\newcommand{\cE}{\mathcal{E}}
\newcommand{\cW}{\mathcal{W}}

\newcommand{\cV}{\mathcal{V}}

\newcommand{\diag}{{\rm diag}}
\newcommand{\sgn}{{\rm sgn}}
\newcommand{\argmin}{{\rm argmin}}

\newcommand{\R}{\mathbb{R}}
\newcommand{\N}{\mathbb{N}}

\begin{document}
\title{Event-triggered bipartite consensus for multiagent system with general linear dynamics: an integral-type event triggered control}
\author{Nhan-Phu Chung, Thanh-Son Trinh, and  Woocheol Choi
\thanks{  All authors were supported by the National Research Foundation of Korea (NRF) grants funded by the Korea government No. NRF- 2016R1A5A1008055. N.P. Chung was supported by the National Research Foundation of Korea  (NRF) grant funded by the Korea government No. NRF-2019R1C1C1007107. W. Choi was supported by the National Research Foundation
of Korea(NRF) grant funded by the Korea government  No. NRF-
2021R1F1A1059671. }
\thanks{N. P. Chung, T. S. Trinh, and W. Choi  are with the Department of Mathematics, Sungkyunkwan University, Suwon 16419 (e-mail: phuchung@skku.edu; phuchung82@gmail.com; sontrinh@skku.edu; choiwc@skku.edu).}}
\maketitle

\begin{abstract}
In this paper, we propose an integral based event-triggering controller for bipartite consensus of the multi-agent systems whose dynamics are described by general linear system. We prove that the system achieves the bipartite consensus in asymptotic regime and there is a positive minimum inter-event time (MIET) between two consecutive triggering times of each agent. The proof of the asymptotive stability involves a novel argument used to bound the norm of the difference between the true state and its estimated state (for each time) by an integration of its square. 
 Numerical results are provided supporting the effectiveness of the proposed controller.

\end{abstract}

\begin{IEEEkeywords}
bipartite consensus, chattering and genuinely Zeno, integral-based event triggered control, multi-agent systems with linear dynamics. 
\end{IEEEkeywords}

\section{Introduction}
The coordinated control of multi-agent system (MASs) has recieved a lot of interest in the last decades, thanks to its wide applications in various fields containing engineering and computer science, biology, and social sciences. In particular, distributed controls for the coordinates of multi-agent systems (MASs) over graphs have been studied intensively by variety of researchers \cite{A,CYRC,KCM,OM,QGY,QMSW,YCCH, ZFHC}. A fundamental problem in the coordinated control of multi-agent system is the consensus problem which requires all agents to asymptotically achieve a common quantity during their cooperative interactions. The main task of the consensus problem in distributed multi-agent systems is to design
the control so that all agents will obtain agreement using only interactions of neighbors.

 In certain practical problems, some agents collaborate while others are in another competitive group. These systems are represented by signed graphs and the weight of each edge will be positive/negative if the two agents are cooperative/competitive. 
 {The agents are said to exhibit a bipartite consensus if they reach agreement in modulus but not in sign \cite{A}. The distributed Laplacian-like control schemes were developed in \cite{A} for the bipartite consensus of single-integrator agents, and extended to directed signed graph \cite{HZ, MDJ}, general linear systems \cite{ZC, Z, VM}. Also, the bipartite consensus problems have been studied for general linear system with input-saturation \cite{QFZG} and communication nosies \cite{HWLG}.

Due to limitations of sources in the multi-agent systems, we can not assume that agents have continuous access to others' states. Therefore, agents in the system should have strategy to take various actions in automatically schedules instead of doing so continuously. As an effective solution for the scheduling, event-triggered control designs have been developed for MASs. For event-triggered control systems, controller updates are activated only when a sutiably designed event-triggering condition is satisfied.
The event-triggered consensus problem of MASs is first studied by Dimarogonas, Frazzoli and Johansson \cite{DFJ} and has been received tremendous attention after that. Numerous distributed event-triggered consensus protocols were introduced for MASs of first-order \cite{SDJ,SHAD, WMXLSW,YYWJ}, of second-order \cite{LLHZ}, and of linear dynamics \cite{HLF,ZJF}.


The distributed control with event-triggered communications have been applied to the bipartite consensus problem for the single-integrator agents \cite{LCHX, YCCH}, the double-integrator systems \cite{RSL}, the general linear systems \cite{CZDZ, ZCH}, and the heterogeneous systems \cite{WLJH}. In addition, the distributed control with event-triggered communications were studied in the context of the bipartite consensus for the system with input time delay \cite{CZZH} and the prescribed-time bipartite consensus problem \cite{CYH}.
}
\

There are two common issues in designing the event-triggered controls for the consensus problems. One is to guarantee that the controlled system actually achieves the consensus in asymptotic regime or in a prescribed time. The other one is to find a lower bound for the difference between two consecutive triggering times. The larger the lower bound is, the better the system is in saving the energy for the communications. These two issues could be in conflict with each other, and so it is important to design a balanced controller satisfying both the two properties. In particular, it is a non-trivial issue to prove that the controlled system has  a  positive minimum inter-event time (MIET) between two consecutive triggering times while the system achieves the consensus. The works \cite{CU, DL, GCC} designed event-triggered controllers which guarantee  a positive MIET and the consensus is achieved up to a small error. 

The idea of integral-based event-triggering condition yielding larger inter-event intervals was proposed by Mousavi, Ghodrat and Marquez \cite{MGM}. Later, this idea is applied to multi-agent systems in \cite{GM} and it has been investigated further in \cite{MZ,WM,ZJF,ZLW, ZW}. In particular, Zhang, Lunze and Wang \cite{ZLW} showed that the proposed controller system for the consensus problem has a positive MIET and the consensus is achieved asymptotically without a small error. This kind of idea was not yet applied to bipartite consensus for general linear system.  In this article, we propose a novel integral-based event triggering condition for our MASs with linear dynamics for both the bipartite consensus problem which admits the asymptotic bipartite consensus and a positive MIET.

We mention that it is more difficult to prove the consensus  for controls based on integral based event-triggering than that based on the point-wise event-triggering, since the Lyapunov stability analysis should be performed after a time integration (see \eqref{F-temp 5}) in order to use the triggering condition. 
In addition, for applying Barbalat's lemma \cite[Lemma 8.2]{K}, one needs to obtain a uniform bound for the difference $\|\hat{z}_i (t) - z_i (t)\|$ between the true state and the estimated state. We achieve this uniform bound by combining a bound on integration of  $\|\hat{z}_i (t) - z_i (t)\|^2$ on an interval and a regularity estimate based obtained in Lemma \ref{L-boundedness of solutions of non-control systems}.

One more subtle issue in the event-triggered controls is the chattering Zeno problem, which means that $t_{k+1}^i = t_k^i$ where $t_k^i$ is the $k$-th event-triggering instant of agent $i$ (see \eqref{E-event triggered times}). To avoid this issue, we add an exponentially decaying term and show that it does not interrupt the system to achieve the bipartitle asymptotic consensus. Our contributions are summarized as follows.


\textit{Statement of contributions:} 
\begin{enumerate}
\item We design an integral  based 
event-triggered control for the bipartite consensus of general linear multi-agent systems, which is motivated by the control scheme proposed in \cite{ZLW} for consensus of MASs. We add an exponentially decaying term in the triggering function in order to exclude the chattering Zeno behavior.  We remark that the previous works \cite{CL,GCC,HLF,LWDR,YRLC,ZLW,ZJF} implicitly assumed the systems do not have chattering Zeno behavior to establish the Zeno-freeness or the positive MIET of the systems.

\item We show that the bipartite consensus is achieved with our proposed event-triggered control. 
One essential step in proving the consensus using integral based event-triggered control is the application of Barbalat's lemma. This step requires obtaining a uniform bound for the derivative $\dot{V}$ of the a Lyapunov function $V$ used in the proof (\cite{MZ, WM,ZLW, ZJF,ZW}).  We give a full detail for obtaining the uniform boundedness of the derivative (Lemma \ref{L-boundedness of the error function}). 
 
\item We extend the work of \cite{YCCH} on Zeno-free analysis on event-triggered bipartite consensus for single-integrator multiagent systems to general linear dynamics. In addition to prove that the Zeno behavior is excluded, we obtain a uniform lower bound for the inter-event intervals for any agent.   
\end{enumerate}
Our paper is organized as follows. In section II, we review our notations, basic definitions and properties of graphs, and present our problem formulation of bipartite event-triggered consensus of MASs with general linear dynamics. Our main results will be presented in section III. In section IV, we will illustrate the efficiency and feasibility of our results by a numerical example. And in the last section, our conclusions are given.  
\section{Preliminaries}
\subsection{Notations} Let $\R$ be the set of all real numbers and $\N$ be the set of all nonnegative integer numbers. We denote $|s|$ the absolute value of a number $s\in \R$. The Euclidean norm of a vector $x\in \R^n$ is denoted by $\|x\|$. We denote $M_{m\times n}(\R)$ the set of all matrices of $m$ rows and $n$ columns with real valued entries, and when $m=n$ we write $M_n(\R)$ instead. The transpose matrix of a matrix $M\in M_{m\times n}(\R)$ is denoted by $M^T$. Given $M\in M_{m\times n}(\R)$, we denote by $\|M\|$ its matrix operator norm. For every $N\in \N$ we denote by $1_N$ the vector with $N$ columns and all entries are one. The sign function is denoted by $\sgn(\cdot)$. If $A\in M_{m\times n}(\R)$ and $B\in M_{p\times q}(\R)$ then $A\otimes B\in M_{pm\times qn}(\R)$ is the Kronecker product. 
\subsection{Graphs}
Let $\cG=(\cV,\cE,\cW)$ be a signed graph with a set of vertices $\cV=\{1,\dots, N\}$, an edge set $\cE\subset \cV\times \cV$ and an adjacency matrix $\cW\in M_N(\R)$ of the signed weights of $\cG$ with $w_{ij}\neq 0$ if $(j,i)\in \cE$ and $w_{ij}=0$ otherwise. We always assume that the graph has no self-loops, i.e. $w_{ii}=0$ for every $i\in \cV$. The Laplacian matrix $L$ of $\cG$ is defined by $L=C-\cW$, where $C=\diag\bigg( \sum_{j=1}^N|w_{1j}|, \sum_{j=1}^N|w_{2j}|,\cdots, \sum_{j=1}^N|w_{Nj}|\bigg)$.

A connected signed graph $\cG$ is \textit{structurally balanced} if there exists a bipartition $\cV_1, \cV_2$ of nodes such that $w_{ij}\geq 0$ for $i,j\in \cV_k$, $k\in \{1,2\}$, and $w_{ij}\leq 0$ for $i\in \cV_k,j\in \cV_l$, $l\neq k$, $l,k\in \{1,2\}$.

\begin{defn}(\cite[Definition 6]{YCCK}) . 
For a connected graph $\cG=(\cV,\cE,\cW)$ with the Laplacian matrix $L$, the general algebraic connectivity is defined by
\label{L-Algebraic connectivity}
$$\alpha(L)=\min_{y^T\xi=0,y\neq 0}\dfrac{y^TL  y}{y^Ty}.$$
Furthermore $\alpha(L)=\lambda_2(L)$, where $\lambda_2(L)$ is the smallest positive eigenvalue of the Laplacian matrix $L$.
\end{defn}
\begin{lem}
\label{L-signed matrix}
(\cite[Lemma 1]{A})
Let $\cG=(\cV,\cE,\cW)$ be a connected, structurally balanced signed graph. Then there exists a matrix $D=\diag(\sigma_1,\sigma_2,\dots, \sigma_N)$ such that entries of $D\cW D$ are all nonnegative, where $\sigma_i\in \{1,-1\}$ for $i\in \cV$.
\end{lem}

\subsection{Problem formulation}
Let $A\in M_n(\R)$ and $B\in M_{n\times m}(\R)$. Consider a group of $N$ agents with linear dynamics under the graph $\cG=(\cV,\cE,W)$.
\begin{align}
\label{E-Original}
\dot{x}_i(t)=Ax_i(t)+Bu_i(t), \mbox{ } i\in \cV=\{1,\dots, N\},
\end{align}
where $x_i(t)\in \R^n$ and $u_i(t)\in \R^m$ are the state and the control input of agent i, respectively. We assume that the signed graph $\cG=(\cV,\cE,\cW)$ is structurally balanced and connected.

Every agent $i$ activates a non-decreasing triggering sequence $0=t^i_0\leq t_1^i\leq t^i_2\leq \dots$. For any state $x_i(t)$ of agent $i$, we define $\hat{x}_i (t): = e^{(t-t_k^i)A}x_i (t_k^i)$, where $k:=\argmin\{t-t_\ell:t_\ell\leq t, \ell\in \N\}$. We define the error measurements $e_i(t):=\hat{x}_i(t)-x_i(t)$. We denote $x(t)=[x_1^T(t),\dots,x_N^T(t)]^T$, $\hat{x}(t)=[\hat{x}_1^T(t),\dots,\hat{x}_N^T(t)]^T$, $e_x(t)= [e_1^T(t),\dots,e_N^T(t)]^T$.

The event-triggered consensus protocol is given as
\begin{align}
\label{E-event-triggered consensus protocol}
u_i(t)&=-K\sum_{j=1}^N|w_{ij}|(\hat{x}_i(t)-\sgn(w_{ij})\hat{x}_j(t)),
\end{align}
where $K$ is the control matrix which is designed later.

Using Kronecker product, we can write \eqref{E-Original} as
\begin{equation*}
\label{E-temp of x}
\dot{x}(t)=(I_N\otimes A-L\otimes BK)x(t)-(L\otimes BK)e_x(t).
\end{equation*}
The system of agents is said to be \textit{bipartite consensus} if $\lim_{t\to \infty}\|\sigma_ix_i(t)-\sigma_jx_j(t)\|=0$ for every $1\leq i,j\leq N$.
\begin{defn}
We say that an agent $i$ has \textit{Zeno behavior} if there exists $t_\infty\in \R$ such that $\lim_{k\to \infty}t^i_k=\sum_{k=0}^\infty t_{k+1}^i-t^i_k=t_\infty$. There are two different types of Zeno behavior \cite{AS} as follows. For an agent $i$ which has Zeno behavior, agent $i$ has
\begin{enumerate}
\item
\textit{Chattering Zeno} if there exists $M\in \N$ such that $t^i_{k+1}=t^i_k$ for every $k\geq M$. 
\item \textit{Genuinely Zeno} if $t^i_{k+1}-t^i_k>0$ for every $k\in \N$.
\end{enumerate} 

The system excludes Zeno behavior if all agents do not have Zeno behavior.
\end{defn}
We need to find a triggering condition for each agent $i$ such that the system achieves both bipartite consensus and Zeno-freeness.

\section{Main results} 
In this section, we will present an event-triggered design such that the system \eqref{E-Original} achieves bipartite consensus. Let us introduce some intermediate variables.

We denote $z_i(t):=\sigma_ix_i(t)$, $\hat{z}_i (t) = \sigma_i \hat{x}_i (t)$, $e_{z,i}(t)=\hat{z}_i(t)-z_i(t)$, $z(t)=[z_1^T(t),\dots,z_N^T(t)]^T$, $\hat{z}(t)=[\hat{z}_1^T(t),\dots,\hat{z}_N^T(t)]^T$, $e_z(t)= [e_{z,1}^T(t),\dots,e_{z,N}^T(t)]^T$. 

We define a new Laplacian matrix $L_D:=DLD$, where $D$ is the matrix given in Lemma \ref{L-signed matrix}. Then from \eqref{E-temp of x} we get
\begin{equation*}
\label{E-dynamics of z}
\dot{z}(t)=(I_N\otimes A-L_D\otimes BK)z(t)-(L_D\otimes BK)e_z(t).
\end{equation*}

Let $h:[0,\infty)\to (0,\infty)$ be a continuous function satisfying that $\int_0^\infty h(s)ds=1$, for example choosing $h(t)=\alpha^{-t}$ for some $\alpha>1$. 
We define the triggering functions as
\begin{align}
\label{E-inegral triggering function}
f_i^k(t):=\int_{t_k^i}^t g_i(s)ds,
\end{align}
where $g_i(s)=\|e_{z,i}(s)\|^2-\beta_i\|\sum_{j=1}^N|w_{ij}|(\hat{z}_i(s)-\hat{z}_j(s))\|^2-\beta_ih(s)$ and $\beta_i$ are constants which will be determined later.

The triggering instants for agent $i$ are defined by 
\begin{align}
\label{E-event triggered times}
t^i_{k+1}:=\inf\{t>t^i_k:f^k_i(t)>0\}.
\end{align}

To our knowledge, the following lemma is the first result in literature of proving the freeness of chattering Zeno for multiagent systems without adding a minimum positive MIET for triggering instants.
\begin{lem} For every $i\in\{1,\ldots,N\}$ and $k\in\N$ we have $t^i_{k+1}>t^i_{k}$.
\end{lem}
\begin{proof}
For $t\geq t^i_k$ we define $g(t):=\|e_{z,i}(t)\|^2-\beta_ih(t).$ Then $g(t^i_k)=-\beta_ih(t^i_k)<0.$
As the continuity of $g$ we get that there exists $\varepsilon>0$ such that for every $t\in [t^i_k,t^i_k+\varepsilon]$ one has $g(t)\leq -\frac{1}{2}\beta_ih(t^i_k).$ Hence, for every $t\in [t^i_k,t^i_k+\varepsilon] $ we get that 
\begin{align*}
f^k_i(t)&=\int_{t^i_k}^tg_i(s)ds\\
&\leq\int^t_{t^i_k}g(s)ds\\
&\leq -\frac{1}{2}\beta_ih(t^i_k)(t-t^i_k).
\end{align*} 
This yields, $f^k_i(t^i_k+\varepsilon)\leq -\frac{1}{2}\beta_ih(t^i_k)\varepsilon<0.$ By the definition of $t^i_{k+1}$, we obtain that $t^i_{k+1}>t^i_k$.
\end{proof}
Now we provide some standard assumptions on the matrices $A,B$ and the signed graph $\cG=(\cV,\cE,\cW)$. 
\begin{assum}
The pair $(A,B)$ stabilizable, i.e there exists $K_1\in M_{n}(\R)$ such that all eigenvalues of $A+K_1B$ has strictly negative real part. We also assume that the graph $\cG$ is structurally balanced and connected. 
\end{assum} 
As $(A,B)$ is stabilizable, applying \cite[Theorem 9.5]{Zab}, there exists a positive matrix $P\in M_n(\R)$ and $\kappa>0$ such that
\begin{align}
\label{F-LMI of P}
A^TP+PA-\alpha(L)PBB^TP <-\kappa I_n.
\end{align}
We choose the control matrix $K:=B^TP$ in \eqref{E-event-triggered consensus protocol}. We define $\bar{z}(t):=\dfrac{1}{N}\sum_{i=1}^Nz_i(t)$. Then $\dot{\bar{z}}(t)=A\bar{z}(t)$. To show this, we observe that
\begin{equation*}
\begin{split}
\dot{z}_i (t) & = \sigma_i \dot{x}_i (t)
\\
& = \sigma_i A x_i (t) + \sigma_i B u_i (t)
\\
& = Az_i (t) - \sigma_i BK \sum_{j=1}^N |w_{ij}| ( \hat{x}_i (t) - \textrm{sgn}(w_{ij}) \hat{x}_j (t))
\\
& = A z_i (t) - BK \sum_{j=1}^N |w_{ij}| (\hat{z}_i (t) - \hat{z}_j (t)).
\end{split}
\end{equation*}
Summing up this for $1 \leq i \leq N$ we find $\dot{\bar{z}}(t) = A \bar{z}(t)$.

Next we note that the difference $e_i (t) = \hat{x}_i (t) -x_i (t)$ satisfies
\begin{equation*}
\dot{e}_i (t) = A (\hat{x}_i (t) - x_i (t)) - B u_i (t) = A e_{i} (t) - Bu_i (t).
\end{equation*}
Since $e_{z,i}(t) = \sigma_i e_i (t)$ we have
\begin{equation}\label{eq-3-1}
\begin{split}
\dot{e}_{z,i} (t)& = \sigma_i \dot{e}_i (t)
\\
&=\sigma_i (Ae_i (t) - Bu_i (t))
\\
& =A e_{z,i} (t) - \sigma_i Bu_i (t).
\end{split}
\end{equation}
We also note that
\begin{equation*}
\begin{split}
u_i (t)&  = -K \sum_{j=1}^N |w_{ij}| (\sigma_i \hat{z}_i (t) - \sigma_i \hat{z}_j (t))
\\
& = - K \sigma_i \sum_{j=1}^N |w_{ij}| (\hat{z}_i (t) - \hat{z}_j (t)).
\end{split}
\end{equation*}
Inserting this to \eqref{eq-3-1} we find
\begin{equation}\label{eq-3-2}
\dot{e}_{z,i} (t) = Ae_{z,i} (t) + K \sum_{j=1}^N |w_{ij}| (\hat{z}_i (t) - \hat{z}_j (t)).
\end{equation}

We set $\delta_i (t) = z_i (t) - \bar{z}(t)$ and $\delta (t) = [\delta_1 (t)^T, \cdots, \delta_N (t)^T]^T$. Also we define the Lyapunov function $V(t):=\delta^T(t)(I_N\otimes P)\delta(t)$.
\begin{lem}
\label{L-decreasing of the Lyapunov function}
Let $c>0$ such that $\beta:=\kappa-c\|L_D\otimes PBB^TP\|^2>0.$ Put $\beta_{max}:=\max\{\beta_1,\dots, \beta_N\}$. Assume that $\beta_{max}<\min\bigg\{\dfrac{1}{2\|L_D\|^2}, \dfrac{\beta c}{2\|L_D\|^2(\beta c+1)}\bigg\}$. Then for every $t\geq 0$ we get that
\begin{align*}
V(t)-V(0)&\leq -\beta\int_{0}^t\|\delta(s)\|^2ds+\dfrac{1}{c}\int_{0}^t\|e_z(s)\|^2ds\\
&=-c_1\int_0^t\|\delta(s)\|^2ds+c_2,
\end{align*}
where  $c_1=\beta- \dfrac{2\beta_{max}\|L_D\|^2}{c(1-2\beta_{max}\|L_D\|^2)}>0$ and $c_2=\dfrac{N\beta_{max}}{c(1-2\beta_{max}\|L_D\|^2)}$.
\end{lem}
\begin{proof} We have
\begin{align*}
&\dot{V}(t)\\
=&\delta^T(t)(I_N\otimes A^T-L_D\otimes K^TB^T)(I_N\otimes P)\delta(t)\\
&-e^T(t)(L_D\otimes K^TB^T)(I_N\otimes P)\delta(t)\\
&+\delta^T(t)(I_N\otimes P)[(I_N\otimes A-L\otimes BK)\delta(t)-(L_D\otimes BK)e_z(t)]\\
=&\delta^T(t)[I_N\otimes (A^TP+PA)]\delta(t)-2\delta^T(t)(L_D\otimes PBB^TP)\delta(t)\\
&-2\delta^T(t)(L_D\otimes PBB^TP)e_z(t).
\end{align*}
From Definition \ref{L-Algebraic connectivity} we get that 
\begin{align*}
\delta^T(t)(L_D\otimes PBB^TP)\delta(t)\geq \alpha(L_D)\delta^T(t)(I_N\otimes PBB^TP)\delta(t).
\end{align*}
Therefore 
\begin{align*}
\dot{V}(t)\leq &\delta^T(t)[I_N\otimes (A^TP+PA-\alpha(L)PBB^TP)]\delta(t)\\
&-2\delta^T(t)(L_D\otimes PBB^TP)e_z(t)\\
\leq& -\delta^T(t)(I_N\otimes\kappa I_n)\delta(t)-2\delta^T(t)(L_D\otimes PBB^TP)e_z(t)\\
\leq& -\kappa\delta^T(t)\delta(t)-2\delta^T(t)(L_D\otimes PBB^TP)e_z(t).
\end{align*}

On the other hand,
\begin{align*}
&-2\delta^T(t)(L_D\otimes PBB^TP)e_z(t)\\
&\leq c\delta^T(t)(L_D\otimes PBB^TP)[\delta^T(t)(L_D\otimes PBB^TP)]^T+\dfrac{\|e_z(t)\|^2}{c}\\
&\leq c\|L_D\otimes PBB^TP\|^2\delta^T(t)\delta(t)+\dfrac{\|e_z(t)\|^2}{c}.
\end{align*}
Hence 
\begin{align*}
\dot{V}(t)\leq -\beta\|\delta(t)\|^2+\dfrac{\|e_z(t)\|^2}{c}.
\end{align*}
Therefore
\begin{align}
\label{F-temp 5}
V(t)-V(0)&\leq -\beta\int_{0}^t\|\delta(s)\|^2ds+\dfrac{1}{c}\int_{0}^t\|e_z(s)\|^2ds.
\end{align}
From the event triggered conditions \eqref{E-inegral triggering function} and \eqref{E-event triggered times}, for every $i=1,\dots, N$, $k\in \N$ and $t\in [t_k,t_{k+1})$ we get
\begin{align*}
&\int_{t^i_k}^t\|e_{z,i}(s)\|^2ds\\
\leq &\beta_i\int_{t^i_k}^t\|\sum_{j=1}^Nw_{ij}(\hat{z}_i(s)-\hat{z}_j(s)\|^2ds
+\beta_i\int_{t^i_k}^th(s)ds\\
=&\beta_i\int_{t^i_k}^t\|\sum_{j=1}^Nw_{ij}(z_i(s)-z_j(s)+e_{z,i}(s)-e_{z,j}(s)\|^2ds\\
&+\beta_i\int_{t^i_k}^th(s)ds.
\end{align*}
Hence for every $t\geq 0$ we have
\begin{align*}
&\int_{0}^t\|e_{z,i}(s)\|^2ds\\
\leq &\beta_i\int_{0}^t\|\sum_{j=1}^Nw_{ij}(z_i(s)-z_j(s)+e_{z,i}(s)-e_{z,j}(s)\|^2ds\\
&+\beta_i\int_{0}^th(s)ds\\
\leq &\beta_i\int_{0}^t\|\sum_{j=1}^Nw_{ij}(z_i(s)-z_j(s)+e_{z,i}(s)-e_{z,j}(s)\|^2ds+\beta_i.
\end{align*}
Let $L_{D,i}$ be the ith row of $L_D$ then for every $t\geq 0$ we get
\begin{align*}
\int_{0}^t\|e_{z,i}(s)\|^2ds&\leq 2\beta_i\int_{0}^t(\|(L_{D,i}\otimes I_n)z(s)\|^2\\
+&\|(L_i\otimes I_n)e_z(s)\|^2)ds+\beta_i.
\end{align*}
Therefore we obtain
\begin{align*}
&\int_0^t\|e_z(s)\|^2ds\\
\leq &2\beta_{max}\int_{0}^t(\|(L_D\otimes I_n)z(s)\|^2+\|(L_D\otimes I_n)e_z(s)\|^2)ds\\
&+N\beta_{max}.
\end{align*}
As $(L_D\otimes I_n)\delta(t)=(L_D\otimes I_n)z(t)$, combining with $\beta_{max}<\dfrac{1}{2\|L_D\|^2}$ we obtain
\begin{align}
\int_0^t\|e_z(s)\|^2ds&\leq\dfrac{2\beta_{max}\|L_D\|^2}{1-2\beta_{max}\|L_D\|^2}\int_0^t\|\delta(s)\|^2ds\\
+&\frac{N\beta_{max}}{1-2\beta_{max}\|L_D\|^2}.
\label{F-temp4}
\end{align}
Combining \eqref{F-temp 5} with \eqref{F-temp4} we get that
\begin{align*}
V(t)-V(0)\leq &-\beta\int_0^t\|\delta(s)\|^2ds+\frac{N\beta_{max}}{c(1-2\beta_{max}\|L_D\|^2)}\\
&+\dfrac{2\beta_{max}\|L_D\|^2}{c(1-2\beta_{max}\|L_D\|^2)}\int_0^t\|\delta(s)\|^2ds\\
 =&-c_1\int_0^t\|\delta(s)\|^2ds+c_2.
\end{align*}
\end{proof}
\begin{lem}
\label{L-boundedness of solutions of non-control systems}
Suppose that $\dot{y}(t) = Ay(t)$ for $t \geq t_k^i$. Then we have
\begin{equation*}
\|y(T)\|_2^2 \leq \|y(t_k^i)\|_2^2 + 2\|A\|\int_{t_k^i}^T \|y(t)\|_2^2 dt
\end{equation*}
for all $T \geq t_k^i$, where $\|A\| := \max_{v \in \mathbb{R}^n \setminus \{0\}} \frac{\|Av\|}{\|v\|}$.
\end{lem}
\begin{proof}
We find
\begin{equation*}
\begin{split}
\frac{d}{dt} \|y(t)\|^2 & = 2 \langle y' (t), ~y(t)\rangle
\\
& = 2 \langle A y(t),~y(t)\rangle
\\
&\leq 2 \|A\| \|y(t)\|^2.
\end{split}
\end{equation*}
Integrating this inequality over $[t_k^i, T]$ we obtain 
\begin{equation*}
\int_{t_k^i}^T \frac{d}{dt}\|y(t)\|^2 dt \leq \int_{t_k^i}^T 2\|A\| \|y(t)\|^2 dt.
\end{equation*}
This completes the proof.
\end{proof}
\begin{lem}
\label{L-boundedness of the error function}
For every $i$ the function $\|\hat{z}_i(t)-z_i(t)\|$ is uniformly bounded for $t \geq 0$.
\end{lem}
\begin{proof}
From Lemma \ref{L-decreasing of the Lyapunov function} there exists $M>0$ such that $\|z_i(t)-\bar{z}(t)\|\leq M$ for every $t>0$ and every $i=1,\dots, N$. Therefore to prove $\|\hat{z}_i(t)-z_i(t)\|$ is bounded, it suffices to show that $\|\hat{z}_i(t)-\bar{z}(t)\|$ is bounded.

Since $P$ is a positive define matrix, so is $I_N\otimes P$. Hence, $V(t)=\delta^T(t)(I_N\otimes P)\delta(t)\geq 0$. Therefore, from Lemma \ref{L-decreasing of the Lyapunov function} we get that \begin{align}\label{F-bounded integral of delta}
\int_0^t\|\delta(s)\|^2ds\leq \frac{V(0)+c_2}{c_1}.
\end{align} 
Since \eqref{F-bounded integral of delta} and \eqref{F-temp4}, for every $i\in\{1,\ldots,N\}$, we obtain that \begin{align}\label{F-B1}
\int_0^t\|e_{z,i}(s)\|^2ds\leq \int_0^t\|e_z(s)\|^2ds\leq c_3,
\end{align}
where $c_3:=(2\beta_{\max}\|L_D\|^2(V(0)+c_2))/(c_1(1-2\beta_{\max}\|L_D\|^2))+c_2$. 

Furthermore, for every $t\geq 0$ we have that \begin{align*}
&\int_0^t\|\hat{z}_{i}(s)-\overline{z}(s)\|^2\\
&\leq 2\int_0^t\|\hat{z}_{i}(s)-z_i(s)\|^2ds+2\int_0^t\|z_i(s)-\overline{z}(s)\|^2ds\\
&=2\int_0^t\|e_{z,i}(s)\|^2ds+2\int_0^t\|\delta_i(s)\|^2ds.
\end{align*}
From \eqref{F-bounded integral of delta} and \eqref{F-B1} for every $t\geq 0$, we obtain that 
\begin{align}\label{F-B2}
\int_0^t\|\hat{z}_{i}(s)-\overline{z}(s)\|^2\leq 2\left(c_3+\frac{V(0)+c_2}{c_1}\right).
\end{align}

Assume that $\|\hat{z}_{i}(t)-\overline{z}(t)\|$ is not bounded. Then there exists $M_1>M^2+4\|A\|(c_3+(V(0)+c_2)/c_1)$ such that $\|\hat{z}_{i}(T)-\overline{z}(T)\|>M_1$ for some $T\in [t^{i}_{k},t^{i}_{k+1})$. Next, we define $y(t):=\hat{z}_{i}(t)-\overline{z}(t)$. Then $\dot{y}(t)=Ay(t)$. Therefore, by Lemma \ref{L-boundedness of solutions of non-control systems} and \eqref{F-B2}
\begin{align*}
M_1 &<\|\hat{z}_{i}(T)-\overline{z}(T)\|^2\\
&\leq \|\hat{z}_{i}(t^{i}_{k})-\overline{z}(t^{i}_{k})\|^2+2\|A\|\int_{t^{i}_{k}}^T\|\hat{z}_{i}(t)-\overline{z}(t)\|^2dt\\
&=\|z_i(t^{i}_{k})-\overline{z}(t^{i}_{k})\|^2+2\|A\|\int_{t^{i}_{k}}^T\|\hat{z}_{i}(t)-\overline{z}(t)\|^2dt\\
&\leq M^2+4\|A\|\left(c_3+\frac{V(0)+c_2}{c_1}\right)\\
&<M_1.
\end{align*}
This is a contradiction. Hence, we get that $\|e_{z,i}(t)\|$ is bounded for every $i=1,\ldots,N$.
\end{proof}
\begin{thm}
\label{T-main}
Let a multi-agent system as \eqref{E-Original} such that the pair $(A,B)$ is stabilizable, and the graph $\cG=(\cV,\cE,\cW)$ is connected and structurally balanced. Let $P\in M_n(\R)$ be a positive matrix being a solution of \eqref{F-LMI of P}. We consider the event-triggered consensus protocol as \eqref{E-event-triggered consensus protocol} with the control matrix $K:=B^TP$. The event triggered times are determined by \eqref{E-inegral triggering function} and \eqref{E-event triggered times}. Then the system achieves consensus. Furthermore, the  multi-agent system has a positive minimum inter-event time, i.e. there exists $\tau>0$ such that $t^i_{k+1}-t^i_k\geq \tau$ for every $i=1,\dots, N$, $k\in \N$. 
\end{thm}
\begin{proof}
As $P$ is a positive matrix so is $I_N\otimes P$. Therefore $V(t)\geq 0$ for every $t$. Hence from Lemma \ref{L-decreasing of the Lyapunov function} we get that $\int_0^t\|\delta(s)\|^2ds\leq \dfrac{V(0)}{c_1}$ and therefore $\lim_{t\to \infty}\int_0^t\|\delta(s)\|^2ds$ is finite. Let $\kappa_1$ be the smallest eigenvalue of $I_N\otimes P$. Then $V(0)\geq V(t)=\delta^T(t)(I_N\otimes P)\delta(t)\geq \kappa_1\delta^T(t)\delta(t)$. Hence $\delta(t)$ is bounded. On the other hand, as
\begin{align*}
\dot{\delta(t)}=(I_N\otimes A-L_D\otimes BK)\delta(t)-(L_D\otimes BK)e_z(t),
\end{align*}
applying Lemma \ref{L-boundedness of the error function} we get that $\dot{\delta}(t)$ is also bounded. Applying Barbalat's lemma \cite[Lemma 8.2]{K}, we get that $\lim_{t\to \infty}\delta(t)=0$ and hence $\lim_{t\to \infty}\|z_i(t)-z_j(t)\|=0$ for every $i,j$.

Now we prove that the system has positive minimum inter-event time. We note that
\begin{equation*}
\begin{split}
\frac{d}{dt}\|e_{z,i} (t)\|^2  =&2 \langle e_{z,i} (t),~\dot{e}_{z,i} (t)\rangle
\\
\leq &2 \|e_{z,i} (t)\| \|\dot{e}_{z,i} (t)\|
\\
\leq & 2\|BK\|\|e_{z,i} (t)\| \Big\| \sum_{j=1}^N w_{ij}(\hat{z}_i (t) - \hat{z}_j (t))\Big\|\\
&+2 \|A\|\|e_{z,i} (t)\|^2 
\\
\leq &\Big\| \sum_{j=1}^N w_{ij} (\hat{z}_i (t) - \hat{z}_j (t))\Big\|^2\\
&+(2\|A\| + \|BK\|^2) \|e_{z,i} (t)\|^2,
\end{split}
\end{equation*}
where we used \eqref{eq-3-2} in the second inequality.
We then have
\begin{equation*}
\begin{split}
\|e_{z,i} (t)\|^2 & \leq \int_{t_k^i}^t e^{(2\|A\| + \|BK\|^2)(t-s)} \Big\| \sum_{j=1}^N w_{ij} (\hat{z}_i (s) - \hat{z}_j (s))\Big\|^2 ds
\\
&\leq e^{(2\|A\| + \|BK\|^2) (t-t_k^i)} \int_{t_k^i}^t \Big\|\sum_{j=1}^N w_{ij} (\hat{z}_i (s) - \hat{z}_j (s))\Big\|^2 ds.
\end{split}
\end{equation*}
We rewrite the above inequality by
\begin{equation*} 
\|e_{z,i} (s)\|^2  \leq e^{(2\|A\| + \|BK\|^2) (s-t_k^i)} \int_{t_k^i}^s \Big\|\sum_{j=1}^N w_{ij} (\hat{z}_i (r) - \hat{z}_j (r))\Big\|^2 dr.
\end{equation*}
Integrating this, we find 
\begin{equation*}
\begin{split}
&\int_{t_k^i}^t \|e_{z,i} (s)\|^2 ds\\
&  \leq \int_{t_k^i}^t e^{(2\|A\| + \|BK\|^2) (s-t_k^i)} \int_{t_k^i}^s \Big\| \sum_{j=1}^N w_{ij} (\hat{z}_i (r) - \hat{z}_j (r))\Big\|^2 dr  ds
\\
&\leq \int_{t_k^i}^t e^{(2\|A\| + \|BK\|^2) (s-t_k^i)} \int_{t_k^i}^t \Big\| \sum_{j=1}^N w_{ij} (\hat{z}_i (r) - \hat{z}_j (r))\Big\|^2 dr  ds
\\
& = \Big(\int_{t_k^i}^t e^{(2\|A\| + \|BK\|^2) (s-t_k^i)}ds \Big)\int_{t_k^i}^t \Big\| \sum_{j=1}^N w_{ij} (\hat{z}_i (r) - \hat{z}_j (r))\Big\|^2 dr.
\end{split}
\end{equation*}
From this and the triggering condition, we see that
\begin{equation*}
\int_{t_k^i}^{t_{k+1}^i} e^{(2\|A\| + \|BK\|^2)(s-t_k^i)} ds \geq \beta_i.
\end{equation*}
The proof is done.
\end{proof}
\section{Simulations}
In this section, we present a numerical simulation of the bipartite consensus with event-triggered communication given by \eqref{E-event triggered times}. We consider six agents and take the graph $L$ defined in \cite{YCCH} as 
\begin{equation*}
L =  \begin{pmatrix} 3 & -1 &2 & 0 &0 & 0 \\ -1 & 5 &4 &0 &0 &0 \\ 2 & 4& 8&-2 &0&0 \\ 0 &0& -2&6&-1&-3 \\ 0 &0&0& -1&1& 0 \\ 0&0&0&-3&0&3
    \end{pmatrix}
\end{equation*}
We take $A$ and $B$ in \cite{QFZG} given as
\begin{equation*}
A = \begin{pmatrix} 0 & 1 \\ -1 & 0 \end{pmatrix}\quad B = \begin{pmatrix} 0 & 1 \\ 1 &1 \end{pmatrix}.
\end{equation*}
The matrix $P$ satisfying \eqref{F-LMI of P} is computed as
\begin{equation*}
P=5\times \begin{pmatrix} 0.9862 & 0.0143\\ 0.0143& 0.9212\end{pmatrix}.
\end{equation*}
The initial coordinates of agent $1$ and agent $2$ are selected randomly by a uniform distribution of interval $[0,1]$ while those of other agents are  selected randomly by a uniform distribution of interval $[-1,0]$. We set the constant $\beta_k = 0.008$ for $1\leq k \leq 6$.\\
\begin{figure}[htbp]
\includegraphics[height=6cm, width=9cm]{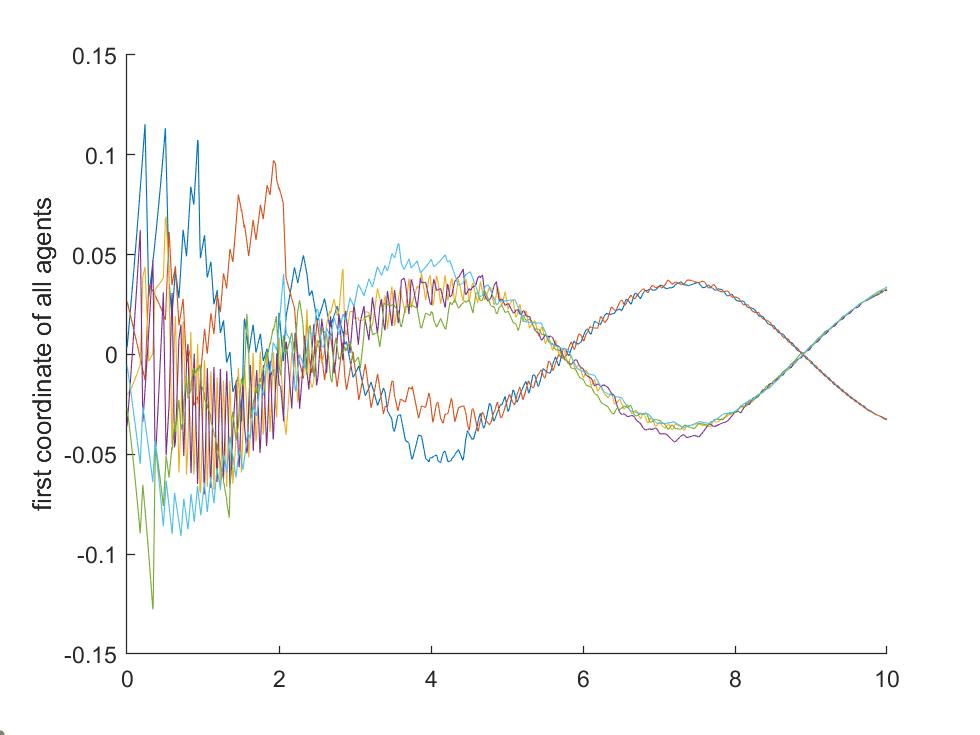}   
\vspace{-0.3cm}\caption{First coordinates of all agents}
\label{fig1}
\end{figure}
\begin{figure}[htbp]
\includegraphics[height=6cm, width=9cm]{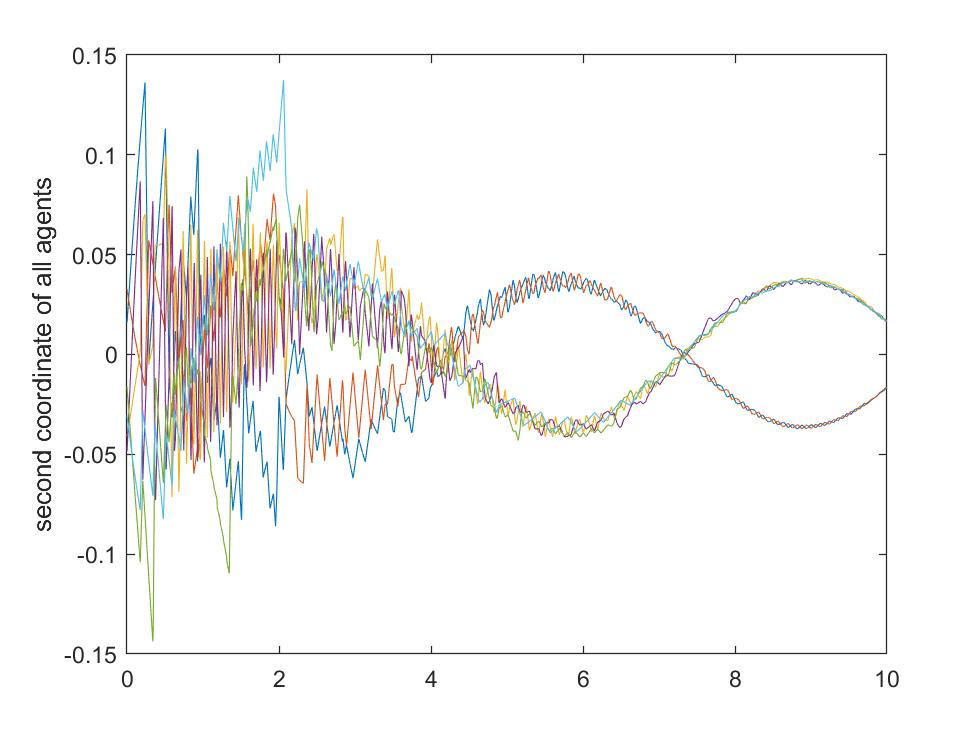}   
\vspace{-0.3cm}\caption{Second coordinates of all agents}
\label{fig2}
\end{figure}
In Figures \ref{fig1} and \ref{fig2} we see that agent 1 and agent 2 converge to a same trajectory and the other agents converge to another same trajectory.
For each $1\leq k$ we let $I_{ev}(k)$ be the minimum interval between two successive event-triggering times of agent $k$. The values of $I_{ev}(k)$ are given in Table \ref{tb3}.
\begin{figure}[htbp]
\includegraphics[height=6cm, width=9cm]{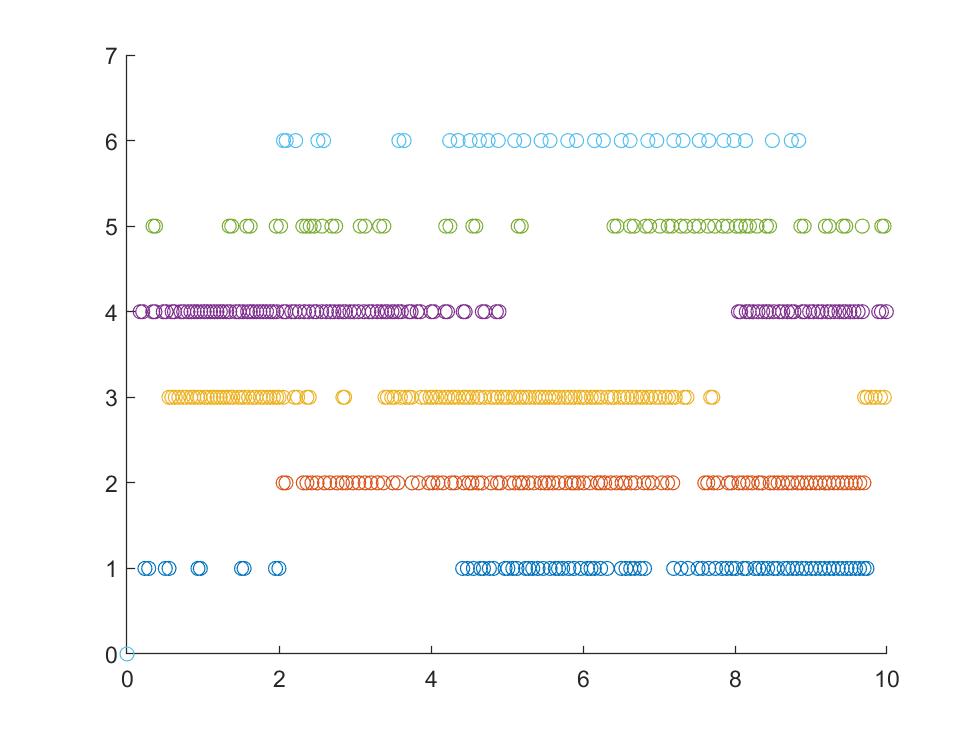}   \caption{Event-triggering instants}
\vspace{-0.3cm}
\label{fig3}
\end{figure}
\

Figure \ref{fig3} indicates the event-triggering times of each agent. The interval lengths between two successive event-triggering times of each agent are given in Figures \ref{fig4} and \ref{fig7}. The ranges of $y$-axis in the diagrams are fixed by $[0, 5* I_{ev}(k)]$ for each agent $1\leq k \leq 6$.
\begin{table}[htp]
\caption{Minimun interval between triggering times}
\begin{center}
\begin{tabular}{c|cccccc}
   $$  k&   1&    2&    3&  4&   5&  6\\[1.0mm]
\hline 
$I_{ev}(k) \,(s)$ &0.031  &   0.037&0.022 & 0.025	    &    0.031& 0.036     \\
\hline
\end{tabular}
\end{center}
\label{tb3}
\end{table}

\begin{figure}[htbp]
\includegraphics[height=6cm, width=9cm]{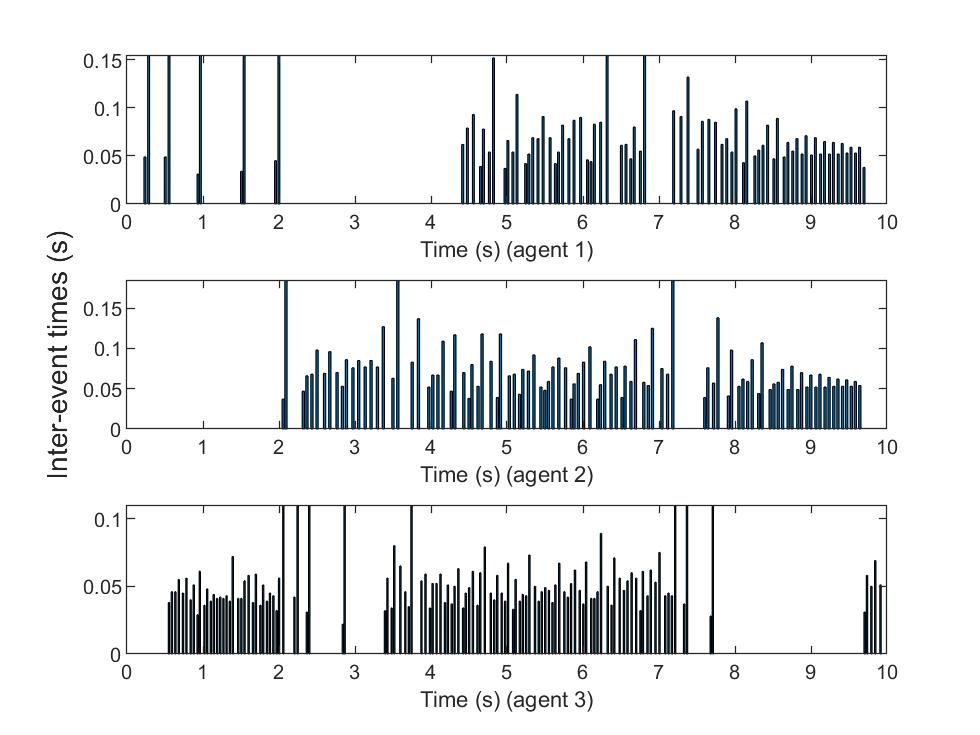}   \caption{Inter-event times of agents 1-3} 
\vspace{-0.3cm}
\label{fig4}
\end{figure}

\begin{figure}[htbp]
\includegraphics[height=6cm, width=9cm]{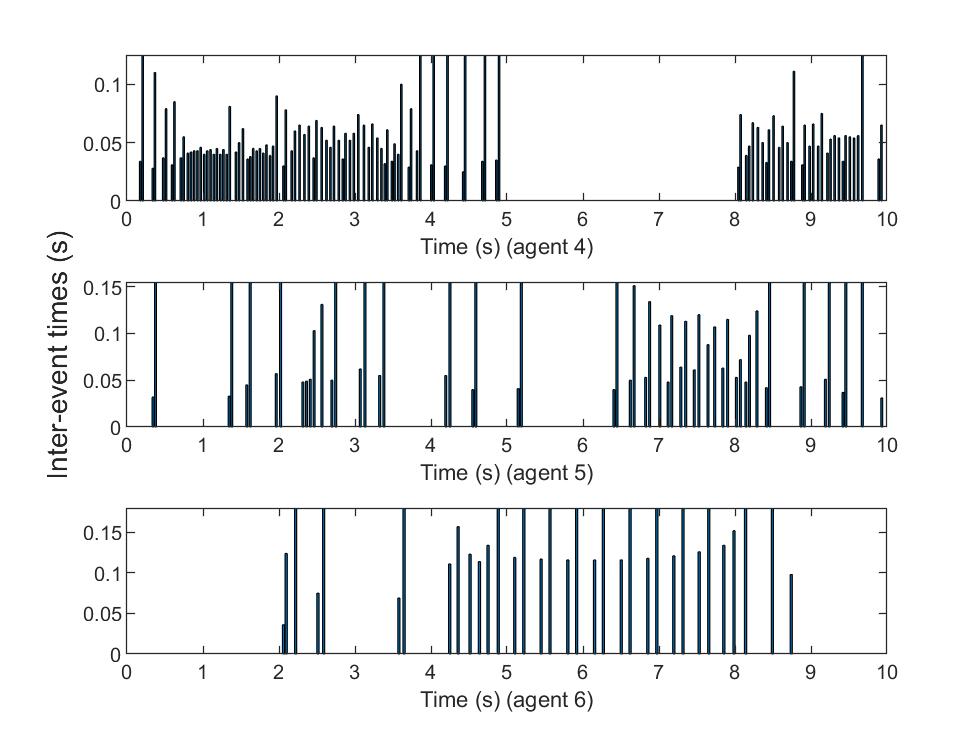}   \caption{Inter-event times of agents 4-6} 
\vspace{-0.3cm}
\label{fig7}
\end{figure}

\section{Conclusion}
In this work, we designed an integral based event-triggered controller for bipartite consensus of the general linear system. We rigorously proved that the bipartite consensus is achieved and there is a positive MIET for the controlled system. The numerical simulation was provided supporting the theoretical results.


\end{document}